\documentclass{amsart}
\usepackage{amssymb,amscd}

\newtheorem{theorem}{Theorem}[section]
\newtheorem{proposition}[theorem]{Proposition}
\newtheorem{corollary}[theorem]{Corollary}
\newtheorem{lemma}[theorem]{Lemma}
\newtheorem{question}[theorem]{Question}
\theoremstyle{definition}

\newtheorem{example}[theorem]{Example}
\newtheorem{construction}[theorem]{Construction}
\theoremstyle{remark}
\newtheorem*{remark}{Remark}

\numberwithin{equation}{section}

\newcommand{\mb}[1]{{\textbf {\textit#1}}}

\def\C{\mathbb C}

\def\R{\mathbb R}
\def\Z{\mathbb Z}

\def\ge{\geqslant}

\def\le{\leqslant}

\newcommand{\Ker}{\mathop{\rm Ker}}

\newcommand{\bin}[2]{{\textstyle\binom{#1}{#2}}}

\newcommand{\OSU}{\varOmega^{\mbox{\scriptsize\textit{SU}}}}
\newcommand{\OU}{\varOmega^U}

\newcommand{\cs}{\mathbin{\#}}

\begin{document}

\title[On toric generators in $U$- and $SU$-bordism rings]{On
toric generators in the unitary and\\ special unitary bordism
rings}

\author{Zhi L\"u}
\address{\noindent School of Mathematical Sciences, Fudan University,
Shanghai, 200433, People's Republic of China}
\email{zlu@fudan.edu.cn}
\urladdr{http://homepage.fudan.edu.cn/zhilu/}

\author{Taras Panov}
\address{Department of Mathematics and Mechanics, Moscow
State University, Leninskie Gory, Moscow, 119991, Russia,
\newline Institute for Theoretical and Experimental Physics,
Moscow, Russia,
\newline Institute for Information Transmission Problems,
Russian Academy of Sciences}
\email{tpanov@mech.math.msu.su}
\urladdr{http://higeom.math.msu.su/people/taras/}

\thanks{L\"u was supported by the NSFC, grants 11371093, 11431009 and 11661131004. Panov was supported by the Russian
Foundation for Basic Research (grants 14-01-00537 and 16-51-55017)
and a grant from Dmitri Zimin's `Dynasty' foundation. Both authors
were supported by the Key Laboratory of Mathematics for
Nonlinear Sciences in Fudan University, Chinese Ministry of
Education.}



\begin{abstract}    
We construct a new family of toric manifolds generating the
unitary bordism ring. Each manifold in the family is the complex
projectivisation of the sum of a line bundle and a trivial bundle
over a complex projective space. We also construct a family of
special unitary quasitoric manifolds which contains polynomial
generators of the special unitary bordism ring with 2 inverted in
dimensions $>8$. Each manifold in the latter family is obtained
from an iterated complex projectivisation of a sum of line bundles
by amending the complex structure to make the first Chern class
vanish.
\end{abstract}

\maketitle


\section{Introduction}
Finding geometric representatives of bordism classes is a
classical problem on the borders of geometry and topology. The
theory of bordism and cobordism is one of the deepest and most
influential parts of algebraic topology, which experienced a
spectacular development in the 1960s. Although the original
definition of bordism, going back to Pontryagin and Thom, was very
geometric, it soon became clear that elaborate homotopy-theoretic,
algebraic and number-theoretic techniques were required to obtain
structural results on bordism groups and (co)bordism rings.

Most calculations of bordism rings of a point for the classical
series of Lie groups were settled by coordinated efforts of many
topologists by the end of the 1960s (with the notable exception of
symplectic bordism, whose structure is still not described
completely). These results were summarised in the monograph by
Stong~\cite{ston68}. Nevertheless, it has remained a challenging
task to describe particular geometric representatives for
generators of bordism rings (which tend to be rings of polynomials
when $2$ is inverted) and other `special' bordism classes. The
importance of this problem was much emphasised in the original
works such as Conner and Floyd~\cite{co-fl66m}.

Over the rationals, the bordism rings are generated by projective
spaces, but the integral generators are more subtle as they
involve divisibility conditions on characteristic numbers. One of
the few general results on geometric representatives for bordism
classes known from the early 1960s is that the complex bordism
ring $\OU$, which is an integral polynomial ring, can be generated
by the so-called \emph{Milnor hypersurfaces} $H(n_1,n_2)$. These
are hyperplane sections of the Segre embeddings of products $\C
P^{n_1}\times\C P^{n_2}$ of complex projective spaces. Similar
generators exist for unoriented and oriented bordism rings.

The early progress was impeded by the lack of examples of
higher-dimensional (stably) complex manifolds for which the
characteristic numbers can be calculated explicitly. With the
appearance of \emph{toric varieties} in the late 1970s and
subsequent development of toric topology (see Buchstaber and Panov~\cite{bu-pa}), a host of
concrete examples of complex manifolds with large symmetry groups
has been produced for which characteristic numbers can be
calculated effectively using combinatorial-geometric techniques.

In~\cite{bu-ra98r}, Buchstaber and Ray constructed a set of
generators for $\OU$ consisting entirely of complex projective
toric manifolds $B(n_1,n_2)$, which are projectivisations of sums
of line bundles over bounded flag manifolds. Later it was shown
in Buchstaber, Panov and Ray~\cite{b-p-r07} that one can get a geometric representative in
\emph{every} complex bordism class if toric manifolds are relaxed
to \emph{quasitoric} ones; the latter still have a `large torus'
action, but are only stably complex instead of being complex.
Characteristic numbers of toric manifolds satisfy quite
restrictive conditions (e.g. their Todd genus is always~1) which
prevent the existence of a toric representative in every bordism
class; quasitoric manifolds enjoy more flexibility. We note that
representing \emph{polynomial} generators of $\OU$ by toric
manifolds remains open; some progress has been made recently
by Wilfong~\cite{wilf}.

Here we consider a family of projective toric manifolds obtained
by iterated projectivisation of sums of line bundles, starting
from a complex projective space. Such iterated projectivisations
are also known as \emph{generalised Bott
manifolds} (see Masuda and Suh~\cite{ma-su08} and Buchstaber and Panov~\cite[\S7.8]{bu-pa}). Our first result
(Theorem~\ref{projcobgen}) shows that the complex bordism ring
$\OU$ can be generated by the most simple nontrivial 2-stage
projectivisations: manifolds $L(n_1,n_2)=\C P(\xi)$, where $\xi$
is the sum of a tautological line bundle and an $n_2$-dimensional
trivial bundle over~$\C P^{n_1}$. This new toric generator set is
somewhat simpler than either of the set of Milnor hypersurfaces
$\{H(n_1,n_2)\}$ or Buchstaber and Ray's toric set
$\{B(n_1,n_2)\}$.

We proceed by providing explicit families of quasitoric
$SU$-manifolds which contain polynomial generators of the
$SU$-bordism ring $\OSU\otimes\Z[\frac12]$ (Theorem~\ref{mainth}).
In fact, our quasitoric $SU$-manifolds are genuinely
indecomposable and indivisible elements in $\OSU$ (integrally,
without inverting any prime), however $\OSU$ is not a polynomial
ring.

We recall that a stably complex (or unitary) manifold $M$ is
\emph{special unitary} (an \emph{$SU$-manifold} for short) if
$c_1(M)=0$. A renewed interest to this class of manifolds has been
stimulated by the development of geometry motivated by physics;
the notion of a \emph{Calabi--Yau manifold} plays a central role
here. By a Calabi--Yau manifold one usually understands a K\"ahler
$SU$-manifold; it has a Ricci flat metric by the theorem of Yau.
We note however that our $SU$-manifolds are rarely K\"ahler.

As was observed by L\"u and Wang in~\cite{lu-wa}, quasitoric $SU$-manifolds can be
constructed by taking iterated complex projectivisations (which
are projective toric manifolds) and then amending the stably
complex structure so that the first Chern class becomes zero. The
underlying smooth manifold of the result is still toric, but the
stably complex structure is not the standard one. Examples of this
sort were known to Conner and Floyd~\cite{co-fl66m},
however the existence of a torus action was not emphasised and
their amended stably complex structures were actually not~$SU$.

Characteristic numbers of $SU$-manifolds satisfy intricate
divisibility conditions. Ochanine's theorem~\cite{ocha81}
asserting that the signature of an $(8k+4)$-dimensional $SU$-manifold
is divisible by~16 is one of the most famous examples. We
therefore find it quite miraculous that polynomial generators for
the $SU$-bordism ring $\OSU$ occur within the most basic families
of examples that one can produce using toric methods: 2-stage
complex projectivisations, and 3-stage projectivisations with the
first stage being just~$\C P^1$. The proof of Theorem~\ref{mainth}
involves calculating the characteristic numbers and checking
various divisibility conditions.
We use both classical and more recent results on binomial
coefficients modulo a prime.

We note also that the existence of large torus actions indicates
possible applications of our examples in the equivariant setting.
Applicability of toric methods in equivariant bordism is currently
being explored (see Buchstaber, Panov and Ray~\cite{b-p-r10}, Buchstaber and 
Panov~\cite[Ch.~9]{bu-pa} and L\"u~\cite{lu}).

\medskip

We thank Peter Landweber and the referee for their most helpful
comments.

\section{Toric and quasitoric manifolds, cohomology and Chern classes}
Here we collect the necessary information about toric varieties
and quasitoric manifolds; the details can be found
in~\cite{bu-pa}.

A \emph{toric variety} is a normal complex algebraic variety~$V$
containing an algebraic torus $(\C^\times)^n$ as a Zariski open
subset in such a way that the natural action of $(\C^\times)^n$ on
itself extends to an action on~$V$. We only consider nonsingular
complete (compact in the usual topology) toric varieties, also
known as \emph{toric manifolds}.

There is a bijective correspondence between the isomorphism
classes of complex $n$-dimensional toric manifolds and complete
regular fans in~$\R^n$. A \emph{fan} is a finite collection
$\Sigma=\{\sigma_1,\ldots,\sigma_s\}$ of strongly convex cones
$\sigma_i$ in $\R^n$ such that every face of a cone in $\Sigma$
belongs to $\Sigma$ and the intersection of any two cones in
$\Sigma$ is a face of each. A fan $\Sigma$ is \emph{regular} if
each of its cones $\sigma_j$ is generated by part of a basis of
the lattice $\Z^n\subset\R^n$ (we choose the standard lattice for
simplicity). In particular, each one-dimensional cone of $\Sigma$
is generated by a primitive vector $\mb a_i\in \Z^n$. A fan
$\Sigma$ is \emph{complete} if the union of its cones is the
whole~$\R^n$.

Projective toric varieties are particularly important. A
projective toric manifold $V$ is defined by a \emph{lattice
Delzant polytope}~$P$. Given a simple $n$-dimensional polytope $P$
with vertices in the lattice~$\Z^n$, one defines the \emph{normal
fan} $\Sigma_P$ as the fan whose $n$-dimensional cones $\sigma_v$
correspond to the vertices $v$ of $P$, and $\sigma_v$ is generated
by the primitive inside-pointing normals to the facets of $P$
meeting at~$v$. The polytope $P$ is \emph{Delzant} precisely when
its normal fan $\Sigma_P$ is regular. The fan $\Sigma_P$ defines a
projective toric manifold~$V_P$. Different lattice Delzant
polytopes with the same normal fan produce different projective
embeddings of the same toric manifold.

Irreducible torus-invariant divisors on~$V$ are the toric
subvarieties of complex codimension~1 corresponding to the
one-dimensional cones of~$\Sigma$. When $V$ is projective, they
also correspond to the facets of~$P$. We assume that there are $m$
one-dimensional cones (or facets), denote the corresponding
primitive vectors by $\mb a_1,\ldots,\mb a_m$, and denote the
corresponding codimension-1 subvarieties by~$V_1,\ldots,V_m$.

\enlargethispage{2\baselineskip}

\begin{theorem}\label{cohomtoric}
Let $V$ be a toric manifold of complex dimension~$n$, with the
corresponding complete regular fan~$\Sigma$. The cohomology ring
$H^*(V;\Z)$ is generated by the degree-two classes $v_i$ dual to
the invariant submanifolds $V_i$, and is given by
\[
  H^*(V;\Z)\cong \Z[v_1,\ldots,v_m]/\mathcal I,\qquad\deg v_i=2,
\]
where $\mathcal I$ is the ideal generated by elements of the
following two types:
\begin{itemize}
\item[(a)] $v_{i_1}\cdots v_{i_k}$ such that $\mb a_{i_1},\ldots,\mb
a_{i_k}$ do not span a cone of~$\Sigma$;
\item[(b)] $\displaystyle\sum_{i=1}^m\langle\mb a_i,\mb x\rangle v_i$, for
any vector $\mb x\in\Z^n$.
\end{itemize}
\end{theorem}

It is convenient to consider the integer $n\times m$-matrix
\begin{equation}\label{Lambdatoric}
  \varLambda=\begin{pmatrix}
  a_{11}&\cdots& a_{1m}\\
  \vdots&\ddots&\vdots\\
  a_{n1}&\cdots& a_{nm}
  \end{pmatrix}
\end{equation}
whose columns are the vectors $\mb a_i$ written in the standard
basis of~$\Z^n$. Then the ideal~(b) of Theorem~\ref{cohomtoric} is
generated by the $n$ linear forms $a_{j1}v_1+\cdots+a_{jm}v_m$
corresponding to the rows of~$\varLambda$.

\begin{theorem}\label{tangenttoric}
There is the following isomorphism of complex vector bundles:
\[
  \mathcal T V\oplus\underline{\C}^{m-n}\cong
  \rho_1\oplus\cdots\oplus\rho_m,
\]
where $\mathcal T V$ is the tangent bundle, $\underline{\C}^{m-n}$
is the trivial $(m-n)$-plane bundle, and $\rho_i$ is the line
bundle corresponding to~$V_i$, with $c_1(\rho_i)=v_i$. In
particular, the total Chern class of~$V$ is given by
\[
  c(V)=(1+v_1)\cdots(1+v_m).
\]
\end{theorem}

\begin{example}
A basic example of a toric manifold is the complex projective
space $\C P^n$. The cones of the corresponding fan are generated
by proper subsets of the set of $m=n+1$ vectors $\mb
e_1,\ldots,\mb e_n,-\mb e_1-\cdots-\mb e_n$, where $\mb
e_i\in\Z^n$ is the $i$th standard basis vector. It is the normal
fan of the lattice simplex $\varDelta^n$ with the vertices at $\bf
0$ and $\mb e_1,\ldots,\mb e_n$. The matrix~\eqref{Lambdatoric} is
given by
\[
  \begin{pmatrix}
  1&0&0&-1\\
  0&\ddots&0&\vdots\\
  0&0&1&-1
  \end{pmatrix}
\]

Theorem~\ref{cohomtoric} gives the cohomology of $\C P^n$ as
\[
  H^*(\C P^n)\cong\Z[v_1,\ldots,v_{n+1}]/(v_1\cdots v_{n+1},
  v_1-v_{n+1},\ldots,v_n-v_{n+1})\cong\Z[v]/(v^{n+1}),
\]
where $v$ is any of the $v_i$. Theorem~\ref{tangenttoric} gives
the standard decomposition
\[
  \mathcal T\C P^n\oplus\underline{\C}\cong
  \bar\eta\oplus\cdots\oplus\bar\eta\qquad\text{($n+1$ summands)},
\]
where $\eta=\mathcal O(-1)$ is the \emph{tautological} (Hopf) line
bundle over $\C P^n$, and $\bar\eta=\mathcal O(1)$ is its
conjugate, or the line bundle corresponding to a hyperplane $\C
P^{n-1}\subset\C P^n$.
\end{example}

\begin{example}\label{projex1}
An example which will be important for our constructions is the
complex projectivisation of a sum of line bundles over projective
space.

Given two positive integers $n_1$, $n_2$ and a sequence of
integers $(i_1,\ldots,i_{n_2})$, consider the projectivisation
$V=\C P(\eta^{\otimes i_1}\oplus\cdots\oplus\eta^{\otimes
i_{n_2}}\oplus\underline{\C})$, where $\eta^{\otimes i}$ denotes
the $i$th tensor power of $\eta$ over $\C P^{n_1}$ when $i\ge0$
and the $i$th tensor power of $\bar\eta$ otherwise. The manifold
$V$ is the total space of a bundle over $\C P^{n_1}$ with fibre
$\C P^{n_2}$. It is also a projective toric manifold with the
corresponding matrix~\eqref{Lambdatoric} given by

\[
  \begin{pmatrix}
  \noalign{\vspace{-1\normalbaselineskip}}
  \multicolumn{3}{c}{\scriptstyle n_1}\\[-5pt]
  \multicolumn{3}{c}{$\downbracefill$}\,\\
  1&0&0&-1                                           & & & &  \\
  0&\ddots&0&\vdots    &&&\textrm{\huge 0} &                  \\
  0&0&1&-1                                           & & & &  \\
   & & &i_1       &                                   1&0&0&-1\\
   &\textrm{\huge 0} & &\vdots &             0&\ddots&0&\vdots\\
   & & &i_{n_2}   &                                  0&0&1&-1 \\[-5pt]
   &&&&\multicolumn{3}{c}{$\upbracefill$\ }\\[-5pt]
   &&&&\multicolumn{3}{c}{\scriptstyle n_2}\\
  \noalign{\vspace{-1\normalbaselineskip}}
  \end{pmatrix}
  \vspace{1\normalbaselineskip}
\]
The polytope $P$ here is combinatorially equivalent to a product
$\varDelta^{n_1}\times\varDelta^{n_2}$ of two simplices.
Theorem~\ref{cohomtoric} gives the cohomology of $V$ as
\[
  H^*(V)\cong\Z[v_1,\ldots,v_{n_1+1},v_{n_1+2},\ldots,v_{n_1+n_2+2}]/\mathcal
  I,
\]
where $\mathcal I$ is generated by the elements
\begin{gather*}
  v_1\cdots v_{n_1+1},\;v_{n_1+2}\cdots
  v_{n_1+n_2+2},\;v_1-v_{n_1+1},\ldots,v_{n_1}-v_{n_1+1},\\
  i_1v_{n_1+1}+v_{n_1+2}-v_{n_1+n_2+2},\ldots,
  i_{n_2}v_{n_1+1}+v_{n_1+n_2+1}-v_{n_1+n_2+2}.
\end{gather*}
In other words,
\begin{equation}\label{cohomproj1}
  H^*(V)\cong\Z[u,v]\big/\bigl(u^{n_1+1},v(v-i_1u)\cdots(v-i_{n_2}u)\bigr),
\end{equation}
where $u=v_{1}=\cdots=v_{n_1+1}$ and $v=v_{n_1+n_2+2}$.
Theorem~\ref{tangenttoric} gives
\begin{equation}\label{cproj1}
  c(V)=(1+u)^{n_1+1}(1+v-i_1u)\cdots(1+v-i_{n_2}u)(1+v).
\end{equation}
If $i_1=\cdots=i_{n_2}=0$, we obtain $V=\C P^{n_1}\times\C
P^{n_2}$.

The same information can be retrieved from the following
well-known description of the tangent bundle and the cohomology
ring of a complex projectivisation.

\begin{theorem}[Borel and Hirzebruch {\cite[\S15]{bo-hi58}}]\label{projd}
Let $p\colon \C P(\xi)\to X$ be the projectivisation of a complex
$n$-plane bundle $\xi$ over a complex manifold~$X$, and let
$\gamma$ be the tautological line bundle over~$\C P(\xi)$. Then
there is an isomorphism of vector bundles
\[
  \mathcal T\C P(\xi)\oplus\underline{\C}\cong
   p^*{\mathcal T}\!X\oplus(\bar\gamma\otimes p^*\xi),
\]
where $\underline\C$ denotes a trivial line bundle over~$\C
P(\xi)$. Furthermore, the integral cohomology ring of $\C P(\xi)$
is the quotient of the polynomial ring $H^*(X)[v]$ on one
generator $v=c_1(\bar\gamma)$ with coefficients in $H^*(X)$ by the
single relation
\begin{equation}\label{cntatf}
  v^n+c_1(\xi)v^{n-1}+\cdots+c_n(\xi)=0.
\end{equation}
\end{theorem}

The relation above is just $c_n(\bar\gamma\otimes p^*\xi)=0$.

In the case considered above, $\xi=\eta^{\otimes
i_1}\oplus\cdots\oplus\eta^{\otimes i_{n_2}}\oplus\underline{\C}$
over $X=\C P^{n_1}$. We then have $H^*(X)=\Z[u]/(u^{n_1+1})$ where
$u=c_1(\bar\eta)$, so that~\eqref{cntatf} becomes
$v(v-i_1u)\cdots(v-i_{n_2}u)=0$ and the ring $H^*(\C P(\xi))$
given by Theorem~\ref{projd} is precisely~\eqref{cohomproj1}.
Moreover, the total Chern class of $p^*{\mathcal
T}\!X\oplus(\bar\gamma\otimes p^*\xi)$ is given by~\eqref{cproj1}.
\end{example}

The quotient of the projective toric manifold $V_P$ by the action
of the compact torus $T^n\subset(\C^\times)^n$ is the
polytope~$P$.

A \emph{quasitoric manifold} over a combinatorial simple
$n$-dimensional polytope $P$ is a manifold $M$ of dimension $2n$
with a locally standard action of $T^n$ such that the quotient
$M/T^n$ is homeomorphic, as a manifold with corners, to~$P$. (An
action of $T^n$ on $M^{2n}$ is \emph{locally standard} if every
point $x\in M^{2n}$ is contained in a $T^n$-invariant
neighbourhood equivariantly homeomorphic to an open subset in
$\C^n$ with the standard coordinatewise action of~$T^n$ twisted by
an automorphism of the torus; the orbit space of a locally
standard action is a manifold with corners.) We therefore have a
projection $\pi\colon M\to P$ whose fibres are orbits of the
$T^n$-action.

Not every simple polytope can be the quotient of a quasitoric
manifold. Nevertheless, quasitoric manifolds constitute a much
larger family than projective toric manifolds, and enjoy more
flexibility for topological applications.

If $F_1,\ldots,F_m$ are facets of $P$, then each
$M_i=\pi^{-1}(F_i)$ is a quasitoric submanifold of $M$ of
codimension~2, called a \emph{characteristic submanifold}. The
characteristic submanifolds $M_i\subset M$ are analogues of the
invariant divisors $V_i$ on a toric manifold~$V$. Each $M_i$ is
fixed pointwise by a closed $1$-dimensional subgroup (a subcircle)
$T_i\subset T^n$ and therefore corresponds to a primitive vector
$\lambda_i\in\Z^n$ defined up to a sign. Choosing a direction of
$\lambda_i$ is equivalent to choosing an orientation for the
normal bundle $\nu(M_i\subset M)$ or, equivalently, choosing an
orientation for $M_i$, provided that $M$ itself is oriented. An
\emph{omniorientation} of a quasitoric manifold $M$ consists of a
choice of orientation for $M$ and each characteristic submanifold
$M_{i}$, $1\le i\le m$.

The vectors $\lambda_i$ are analogues of the generators $\mb a_i$
of the one-dimensional cones of the fan corresponding to a toric
manifold~$V$ (or analogues of the normal vectors to the facets of
$P$ when $V$ is projective). However, the $\lambda_i$ need not be
the normal vectors to the facets of $P$ in general.

There is an analogue of Theorem~\ref{cohomtoric} for quasitoric
manifolds:

\begin{theorem}\label{cohomqtoric}
Let $M$ be an omnioriented quasitoric manifold of dimension~$2n$
over a polytope~$P$. The cohomology ring $H^*(M;\Z)$ is generated
by the degree-two classes $v_i$ dual to the oriented
characteristic submanifolds $M_i$, and is given by
\[
  H^*(M;\Z)\cong \Z[v_1,\ldots,v_m]/\mathcal I,\qquad\deg v_i=2,
\]
where $\mathcal I$ is the ideal generated by elements of the
following two types:
\begin{itemize}
\item[(a)] $v_{i_1}\cdots v_{i_k}$ such that $F_{i_1}\cap\cdots\cap F_{i_k}=\varnothing$ in~$P$;
\item[(b)] $\displaystyle\sum_{i=1}^m\langle\lambda_i,\mb x\rangle v_i$, for
any vector $\mb x\in\Z^n$.
\end{itemize}
\end{theorem}

By analogy with \eqref{Lambdatoric}, we consider the integer
$n\times m$-matrix
\begin{equation}\label{Lambdaqtoric}
  \varLambda=\begin{pmatrix}
  \lambda_{11}&\cdots& \lambda_{1m}\\
  \vdots&\ddots&\vdots\\
  \lambda_{n1}&\cdots& \lambda_{nm}
  \end{pmatrix}
\end{equation}
whose columns are the vectors $\lambda_i$ written in the standard
basis of~$\Z^n$. Changing a basis in the lattice results in
multiplying $\varLambda$ from the left by a matrix from
$\mbox{\textit{GL}}\,(n,\Z)$. The ideal~(b) of
Theorem~\ref{cohomqtoric} is generated by the $n$ linear forms
$\lambda_{j1}v_1+\cdots+\lambda_{jm}v_m$ corresponding to the rows
of~$\varLambda$. Also, $\varLambda$ has the property that
$\det(\lambda_{i_1},\ldots,\lambda_{i_n})=\pm1$ whenever the
facets $F_{i_1},\ldots,F_{i_n}$ intersect at a vertex of~$P$.

There is also an analogue of Theorem~\ref{tangenttoric}:

\begin{theorem}\label{tangentqtoric}
For a quasitoric manifold $M$ of dimension $2n$, there is an
isomorphism of real vector bundles:
\begin{equation}\label{TMqtiso}
  \mathcal T M\oplus\underline{\R}^{2(m-n)}\cong
  \rho_1\oplus\cdots\oplus\rho_m,
\end{equation}
where $\rho_i$ is the real $2$-plane bundle corresponding to the
orientable characteristic submanifold $M_i\subset M$, so that
$\rho_i|_{M_i}=\nu(M_i\subset M)$.
\end{theorem}

\section{Unitary bordism}
Here we provide a new set of toric generators for the unitary
bordism ring. The general information about unitary (or complex)
bordism can be found in~\cite{ston68}.

Elements of the unitary bordism ring $\varOmega^U$ are the complex
bordism classes of stably complex manifolds. A \emph{stably
complex manifold} is a pair $(M,c_{\mathcal T})$ consisting of a
smooth manifold $M$ and a \emph{stably complex structure}
$c_{\mathcal T}$, where the latter is determined by a choice of an
isomorphism
\begin{equation}\label{stabcom}
  c_{\mathcal T}\colon\mathcal T M\oplus\underline{\R}^N
  \stackrel\cong\longrightarrow\xi
\end{equation}
between the stable tangent bundle of $M$ and a complex vector
bundle~$\xi$. We omit $c_{\mathcal T}$ in the notation when it is
clear from the context. We denote by $[M]\in\varOmega^U$ the
bordism class of a stably complex manifold~$M$. The sum in
$\varOmega^U$ is the disjoint union, and the product is induced by
the Cartesian product of manifolds. The ring $\varOmega^U$ is
graded by the dimension of manifolds.

A complex manifold $M$ (in particular, a toric manifold) has a
canonical stably complex structure arising from the complex
structure on~$\mathcal T M$. An omniorientation of a quasitoric
manifold $M$ gives it a stably complex structure by means of the
isomorphism of Theorem~\ref{tangentqtoric}, because a choice of
orientation for each real 2-plane bundle $\rho_i$ is equivalent to
endowing it with a complex structure.

\begin{example}
The canonical stably complex structure on $\C P^n$ (as a complex
manifold) is given by the isomorphism
\[
  \mathcal T\C P^n\oplus\underline{\R}^2\cong
  \bar\eta\oplus\cdots\oplus\bar\eta\qquad\text{($n+1$ summands)}.
\]
On the other hand, $\C P^n$, viewed as a quasitoric manifold
over~$\varDelta^n$, has $n+1$ characteristic submanifolds, and
therefore $2^{n+2}$ different omniorientations. Each of these
omniorientations gives rise to a stably complex structure,
obtained by replacing some of the line bundles $\bar\eta$ above
with $\eta$, or by reversing the global orientation. Some of these
stably complex structures are equivalent, of course.
\end{example}

We have $H^*(BU(n))\cong\Z[c_1,\ldots,c_n]$, $\deg c_i=2i$, where
the $c_i$ are the universal Chern characteristic classes. For any
sequence $\omega=(i_1,\ldots,i_n)$ of nonnegative integers, there
is the monomial $c_\omega=c_1^{i_1}\cdots c_n^{i_n}$ of degree
$2\|\omega\|=2\sum_{k=1}^n k\,i_k$ and the corresponding
characteristic class $c_\omega(\xi)$ of a complex $n$-plane
bundle~$\xi$. The corresponding tangential Chern
\emph{characteristic number} of a stably complex manifold $M$ is
defined by $c_\omega[M]=c_\omega(\mathcal T M)\langle M\rangle$.
Here $\langle M\rangle$ is the fundamental homology class of~$M$,
and $\mathcal T M$ is regarded as a complex bundle via the
isomorphism~\eqref{stabcom}. The number $c_\omega[M]$ is assumed
to be zero when $2\|\omega\|\ne\dim M$.

\begin{theorem}
Two stably complex manifold $M$ and $N$ represent the same bordism
classes in~$\varOmega^U$ if and only if their sets of Chern
characteristic numbers coincide.
\end{theorem}

Another important characteristic class is $s_n$. It is defined as
the polynomial in $c_1,\ldots,c_n$ obtained by expressing the
symmetric polynomial $x_1^n+\cdots+x_n^n$ via the elementary
symmetric functions $\sigma_i(x_1,\ldots,x_n)$ and then replacing
each $\sigma_i$ by~$c_i$. Define the corresponding characteristic
number as $s_n[M]=s_n(\mathcal T M)\langle M\rangle$.

The ring $\varOmega^U$ was described by Milnor and Novikov (see
\cite{novi62} and Stong~\cite{ston68}):

\begin{theorem}\label{novikov}
The ring $\varOmega^{U}$ is a polynomial ring on generators in every even
degree:
\[
  \varOmega^U\cong\Z[a_i, i>0],\quad\deg a_i=2i.
\]
The bordism class of a stably complex manifold $M^{2i}$ may be
taken to be the $2i$-dimensional generator $a_i$ if and only if
\[
  s_i[M^{2i}]=\begin{cases}
    \pm1&\text{if}\quad i+1\ne p^s\quad\text{for any prime }p,\\
    \pm p&\text{if}\quad i+1=p^s\quad\text{for some prime $p$ and integer~$s>0$.}
  \end{cases}
\]
\end{theorem}

There is no universal description of connected manifolds
representing the polynomial generators $a_n\in\varOmega^U$.
However, there are known explicit families of manifolds whose
bordism classes generate the whole ring $\varOmega^U$.

The classical family of generators for the ring $\varOmega^U$ consists of
the \emph{Milnor hypersurfaces} $H(n_1,n_2)$. Each $H(n_1,n_2)$ is
a hyperplane section of the Segre embedding $\C P^{n_1}\times \C
P^{n_2}\to\C P^{(n_1+1)(n_2+1)-1}$ and may be given explicitly by
the equation
\[
  z_0w_0+\cdots+z_{n_1}w_{n_1}=0
\]
in the homogeneous coordinates $[z_0:\cdots :z_{n_1}]\in\C
P^{n_1}$ and $[w_0:\cdots :w_{n_2}]\in\C P^{n_2}$, assuming that
$n_1\le n_2$. Also, $H(n_1,n_2)$ can be identified with the
projectivisation $\C P(\zeta)$ of a certain $n_2$-plane bundle
over~$\C P^{n_1}$. The bundle $\zeta$ is not a sum of line bundles
when $n_1>1$, so $H(n_1,n_2)$ is \emph{not} a toric manifold in
this case (see~\cite[\S9.1]{bu-pa}).

Buchstaber and Ray~\cite{bu-ra98r} introduced a family
$B(n_1,n_2)$ of \emph{toric} generators of~$\varOmega^U$. Each
$B(n_1,n_2)$ is the projectivisation of a sum of $n_2$ line
bundles over the \emph{bounded flag manifold}~${\mbox{\it
BF\/}}_{n_1}$. Then $B(n_1,n_2)$ is a toric manifold, because
${\mbox{\it BF\/}}_{n_1}$ is toric and the projectivisation of a
sum of line bundles over a toric manifold is toric.

We have $H(0,n_2)=B(0,n_2)=\C P^{n_2-1}$, so
\[
  s_{n_2-1}[H(0,n_2)]=s_{n_2-1}[B(0,n_2)]=n_2.
\]
Furthermore,
\begin{equation}\label{snmilnor}
  s_{n_1+n_2-1}[H(n_1,n_2)]= s_{n_1+n_2-1}[B(n_1,n_2)]=
  -\bin{n_1+n_2}{n_1}\quad\text{ for $n_1>1$},
\end{equation}
see~\cite[\S9.1]{bu-pa} for the details.

We shall need the following two facts from number theory.

\begin{theorem}[Lucas]\label{lucas}
Let $p$ be a prime, and let
\begin{align*}
  n&=n_0+n_1p+\cdots+n_{k-1}p^{k-1}+n_kp^k,\\
  m&=m_0+m_1p+\cdots+m_{k-1}p^{k-1}+m_kp^k
\end{align*}
be the base $p$ expansions of positive integers $m$ and $n$. Then
\[
  \binom nm\equiv\binom{n_0}{m_0}\binom{n_1}{m_1}\cdots\binom{n_k}{m_k}\mod
  p.
\]
Here the standard convention $\bin mn=0$ if $m<n$ is used.
\end{theorem}

For the proof, see e.g.~\cite[Lemma~2.6]{st-ep62}.

\begin{proposition}\label{gcdbinom}
For any integer $n>0$, we have
\[
  \gcd\biggl\{\binom ni,\;0< i< n\biggr\}=\begin{cases}
    1&\text{if}\quad n\ne p^s\quad\text{for any prime }p,\\
    p&\text{if}\quad n=p^s\quad\text{for some prime $p$ and integer~$s>0$.}
  \end{cases}
\]
\end{proposition}
\begin{proof}
%
Assume $n=p^s$. Then each $\bin ni$ with $0<i<n$ is divisible
by~$p$. On the other hand, $\bin{p^s}{p^{s-1}}$ is not divisible
by~$p^2$, e.g. by Kummer's Theorem.

Now assume $n\ne p^s$. Write the base $p$ expansion
\[
  n=n_0+n_1p+\cdots+n_{k-1}p^{k-1}+n_kp^k,
\]
where we may assume $n_k>0$. Take
\[
  i=n_0+n_1p+\cdots+n_{k-1}p^{k-1}+(n_k-1)p^k.
\]
Then $i\ne0$ as otherwise $n=p^k$. By Theorem~\ref{lucas}, $\bin
ni\equiv n_k\not\equiv0\mod p$.
\end{proof}

The fact that each of the families $\{[H(n_1,n_2)]\}$ and
$\{[B(n_1,n_2)]\}$ generates the unitary bordism
ring~$\varOmega^U$ follows from~\eqref{snmilnor},
Proposition~\ref{gcdbinom} and Theorem~\ref{novikov}.

We proceed to describe another family of toric generators
for~$\varOmega^U$.

\begin{construction}\label{Pn1n2}
Given two positive integers $n_1$, $n_2$, we define the manifold
$L(n_1,n_2)$ as the projectivisation $\C P
(\eta\oplus\underline{\C}^{n_2})$, where $\eta$ is the
tautological line bundle over~$\C P^{n_1}$. This $L(n_1,n_2)$ is a
particular case of manifolds from Example~\ref{projex1}, so it is
a projective toric manifold with the corresponding
matrix~\eqref{Lambdatoric} given by\\
\begin{equation}\label{LambdaPn1n2}
  \begin{pmatrix}
  \noalign{\vspace{-1\normalbaselineskip}}
  \multicolumn{3}{c}{\scriptstyle n_1}\\[-5pt]
  \multicolumn{3}{c}{$\downbracefill$}\,\\
  1&0&0&-1                                           & & & &  \\
  0&\ddots&0&\vdots    &&&\textrm{\huge 0} &                  \\
  0&0&1&-1                                           & & & &  \\
   & & &1       &                                   1&0&0&-1\\
   &\textrm{\huge 0} & &0 &             0&\ddots&0&\vdots\\
   & & &0   &                                  0&0&1&-1 \\[-5pt]
   &&&&\multicolumn{3}{c}{$\upbracefill$\ }\\[-5pt]
   &&&&\multicolumn{3}{c}{\scriptstyle n_2}\\
  \noalign{\vspace{-1\normalbaselineskip}}
  \end{pmatrix}
  \vspace{1\normalbaselineskip}
\end{equation}
The cohomology ring is given by
\begin{equation}\label{cohomproj2}
  H^*\bigl(L(n_1,n_2)\bigr)\cong\Z[u,v]\big/\bigl(u^{n_1+1},v^{n_2+1}-uv^{n_2}\bigr)
\end{equation}
with $u^{n_1}v^{n_2}\langle L(n_1,n_2)\rangle=1$. There is an
isomorphism of complex bundles
\begin{equation}\label{TPn1n2}
  \mathcal T
  L(n_1,n_2)\oplus\underline{\C}^2\cong
  \underbrace{p^*\bar\eta\oplus\cdots\oplus p^*\bar\eta}_{n_1+1}\oplus
  (\bar\gamma\otimes p^*\eta)\oplus
  \underbrace{\bar\gamma\oplus\cdots\oplus\bar\gamma}_{n_2},
\end{equation}
where $\gamma$ is the tautological line bundle over $L(n_1,n_2)=\C
P(\eta\oplus\underline{\C}^{n_2})$. The total Chern class is
\begin{equation}\label{cproj2}
  c\bigl(L(n_1,n_2)\bigr)=(1+u)^{n_1+1}(1+v-u)(1+v)^{n_2}
\end{equation}
with $u=c_1(p^*\bar\eta)$ and $v=c_1(\bar\gamma)$. We also set
$L(n_1,0)=\C P^{n_1}$ and $L(0,n_2)=\C P^{n_2}$, then the
identities \eqref{cohomproj2}--\eqref{cproj2} still hold.
\end{construction}

\begin{lemma}\label{lemma1}
For $n_2>0$, we have
\[
  s_{n_1+n_2}\bigl[L(n_1,n_2)\bigr]=
  \bin{n_1+n_2}0-\bin{n_1+n_2}1+\cdots+(-1)^{n_1}\bin{n_1+n_2}{n_1}+n_2.
\]
\end{lemma}
\begin{proof}
Using~\eqref{cproj2} and~\eqref{cohomproj2} we calculate
\begin{multline*}
  s_{n_1+n_2}\bigl(L(n_1,n_2)\bigr)=(v-u)^{n_1+n_2}+n_2v^{n_1+n_2}\\
  =\bin{n_1+n_2}0v^{n_1+n_2}-\bin{n_1+n_2}1uv^{n_1+n_2-1}+\cdots+
  (-1)^{n_1}\bin{n_1+n_2}{n_1}u^{n_1}v^{n_2}+n_2v^{n_1+n_2}\\
  =\Bigl(\bin{n_1+n_2}0-\bin{n_1+n_2}1+\cdots
  +(-1)^{n_1}\bin{n_1+n_2}{n_1}+n_2\Bigr)u^{n_1}v^{n_2},
\end{multline*}
and the result follows by evaluating at $\langle
L(n_1,n_2)\rangle$.
\end{proof}

\begin{theorem}\label{projcobgen}
The bordism classes $[L(n_1,n_2)]\in\varOmega^U_{2(n_1+n_2)}$
generate the ring~$\varOmega^U$.
\end{theorem}
\begin{proof}
Assuming $[L(n_1,n_2)]=0$ when $n_1<0$, we calculate using
Lemma~\ref{lemma1}:
\begin{multline*}
  s_{n_1+n_2}\bigl[L(n_1,n_2)-2L(n_1-1,n_2+1)+L(n_1-2,n_2+2)\bigr]\\[2pt]
  =(-1)^{n_1-1}\bin{n_1+n_2}{n_1-1}+(-1)^{n_1}\bin{n_1+n_2}{n_1}
  -2(-1)^{n_1-1}\bin{n_1+n_2}{n_1-1}=(-1)^{n_1}\bin{n_1+n_2+1}{n_1}.
\end{multline*}
The result follows from Proposition~\ref{gcdbinom} and
Theorem~\ref{novikov}.
\end{proof}

Theorem~\ref{projcobgen} implies that any unitary bordism class
can be represented by a disjoint union of products of projective
toric manifolds. Products of toric manifolds are toric, but
disjoint unions are not, as toric manifolds are connected. In
bordism theory, a disjoint union may be replaced by a connected
sum, representing the same bordism class. However, connected sum
is not an algebraic operation, and a connected sum of two
algebraic varieties is rarely algebraic. This can be remedied by
appealing to quasitoric manifolds, as explained next. Recall that
an omnioriented quasitoric manifold has an intrinsic stably
complex structure, arising from the isomorphism of
Theorem~\ref{tangentqtoric}. One can form the equivariant connected
sum of quasitoric manifolds, as explained in Davis and Januszkiewicz~\cite{da-ja91}, but
the resulting invariant stably complex structure does not
represent the cobordism sum of the two original manifolds. A more
intricate connected sum construction is needed, as outlined below.
The details can be found in~\cite{b-p-r07} or~\cite[\S9.1]{bu-pa}.

\begin{construction}\label{diamondsum}
The construction applies to two omnioriented $2n$-dimensional
quasitoric manifolds $M$ and $M'$ over $n$-polytopes $P$ and $P'$
respectively. The connected sum will be taken at the fixed points
of $M$ and $M'$ corresponding to vertices $v\in P$ and $v'\in P'$.
We need to assume that $v$ is the intersection of the first $n$
facets of $P$, i.e. $v=F_1\cap\cdots\cap F_n$, and the
corresponding characteristic matrix~\eqref{Lambdaqtoric} of~$M$ is
in the \emph{refined form}, i.e.
\[
\varLambda=\left(I\;|\;\varLambda_\star\right)\;=\;\begin{pmatrix}
  1&0&0&\lambda_{1,n+1}&\ldots&\lambda_{1,m}\\
  0&\ddots&0&\vdots&\ddots&\vdots\\
  0&0&1&\lambda_{n,n+1}&\ldots&\lambda_{n,m}
\end{pmatrix}
\]
where $I$ is the unit matrix and $\varLambda_\star$ is an
$n\times(m-n)$-matrix. The same assumptions are made for $M'$,
$P'$, $v'$ and $\varLambda'$.

The next step depends on the \emph{signs} of the fixed points,
$\sigma(v)$ and $\sigma(v')$. The sign of $v$ is determined by the
omniorientation data; it is $+1$ when the orientation of $\mathcal
T_v M$ induced from the global orientation of $M$ coincides with
the orientation arising from
$\rho_1\oplus\cdots\oplus\rho_n|_{v}$, and is $-1$ otherwise.

If $\sigma(v)=-\sigma(v')$, then we take the connected sum $M\cs
M'$ at $v$ and $v'$. It is a quasitoric manifold over $P\cs P'$
with the characteristic matrix
$\left(\varLambda_\star\;|\;I\;|\;\varLambda'_\star\right)$.

If $\sigma(v)=\sigma(v')$, then we need an additional connected
summand. Consider the quasitoric manifold $S=S^2\times\cdots\times
S^2$ over the $n$-cube~$I^n$, where each $S^2$ is the quasitoric
manifold over the segment $I$ with the characteristic matrix
$(1\;1)$. It represents zero in $\varOmega^U$, and may be thought
of as $\C P^1$ with the stably complex structure given by the
isomorphism $\mathcal T\C
P^1\oplus\underline{\R}^2\cong\bar\eta\oplus\eta$. The
characteristic matrix of $S$ is therefore $(I\;|\;I)$. Now
consider the connected sum $M\cs S\cs M'$. It is a quasitoric
manifold over $P\cs I^n\cs P'$ with the characteristic matrix
$\left(\varLambda_\star\;|\;I\;|\;I\;|\;\varLambda'_\star\right)$.

In either case, the resulting omnioriented quasitoric manifold
$M\cs M'$ or $M\cs S\cs M'$ with the canonical stably complex
structure represents the sum of bordism classes
$[M]+[M']\in\varOmega^U_{2n}$.
\end{construction}

The conclusion, which can be derived from the above construction
and any of the toric generating sets $\{B(n_1,n_2)\}$ or
$\{L(n_1,n_2)\}$ for~$\varOmega^U$, is as follows:

\begin{theorem}[\cite{b-p-r07}]\label{6.11}
In dimensions $>2$, every unitary bordism class contains a
quasitoric manifold, necessarily connected, whose stably complex
structure is induced by an omniorientation, and is therefore
compatible with the torus action.
\end{theorem}

\section{Special unitary bordism}
\subsection*{Basics}
A \emph{special unitary structure} (an \emph{$SU$-structure} for short) on a manifold $M$ is a stably complex structure $c_{\mathcal T}$ with a choice of an $SU$-structure on the complex bundle~$\xi$, see~\eqref{stabcom}. A stably complex manifold $(M,c_{\mathcal T})$ admits an $SU$-structure if and only if $c_1(\xi)=0$.
Bordism classes of $SU$-manifolds form the \emph{special unitary
bordism ring}~$\OSU$.

The ring structure of $\OSU$ is more subtle than that
of~$\varOmega^U$. Novikov~\cite{novi62} described
$\OSU\otimes\Z[\frac12]$ (it is a polynomial ring). The 2-torsion
was described by Conner and Floyd~\cite{co-fl66m}. For the
description of the ring structure in $\OSU$ (which is not a
polynomial ring, even modulo torsion), see~\cite{ston68}. We shall
need the following facts.

{\samepage
\begin{theorem}\
\begin{itemize}
\item[(a)] The kernel of the forgetful map $\OSU\to\OU$
consists of torsion elements.

\item[(b)] Every torsion element in $\OSU$ has order~$2$.

\item[(c)] $\OSU\otimes\Z[{\textstyle\frac12}]$ is a polynomial algebra on
generators in every even degree~$>2$:
\[
  \OSU\otimes\Z[{\textstyle\frac12}]\cong
  \Z[{\textstyle\frac12}][y_i\colon i>1],\quad \deg y_i=2i.
\]
\end{itemize}
\end{theorem}
}

For the further analysis of the ring $\OSU$ we need to consider an
auxiliary ring $\mathcal W$, apparently named after
C.\,T.\,C.~Wall. We describe it following~\cite{co-fl66m}
and~\cite{ston68}.

Let $\partial\colon\OU_{2n}\to\OU_{2n-2}$ be the homomorphism
sending a bordism class $[M^{2n}]$ to the bordism class
$[V^{2n-2}]$ of a submanifold $V^{2n-2}\subset M$ dual to
$c_1(M)$. There is a line bundle $\gamma$ over $M$ corresponding
to $c_1(M)$, and the restriction of $\gamma$ to $V$ is the normal
bundle $\nu(V\subset M)$. The stably complex structure on~$V$ is
defined via the isomorphism $\mathcal T M|_V\cong\mathcal T
V\oplus\nu(V\subset M)$. Then $V$ is an $SU$-manifold, so
$\partial^2=0$. The homomorphism $\partial$ is not a derivation of
$\OU$ though; it satisfies the identity
\[
  \partial(a\cdot b)=a\cdot\partial b+\partial a\cdot b-
  [\C P^1]\cdot\partial a\cdot\partial b.
\]

Let $\mathcal W_{2n}$ be the subgroup of $\OU_{2n}$ consisting of
bordism classes $[M^{2n}]$ such that every Chern number of
$M^{2n}$ of which $c_1^2$ is a factor vanishes. The forgetful
homomorphism decomposes as $\OSU_{2n}\to\mathcal
W_{2n}\to\OU_{2n}$, and the restriction of the boundary
homomorphism $\partial\colon\mathcal W_{2n}\to\mathcal W_{2n-2}$
is defined.

The direct sum $\mathcal W=\bigoplus_{i\ge0}\mathcal W_{2i}$ is
\emph{not} a subring of~$\OU$: one has $[\C P^1]\in\mathcal W_2$,
but $c_1^2[\C P^1\times\C P^1]=8\ne0$, so $[\C P^1]\times[\C
P^1]\notin\mathcal W_4$. However, $\mathcal W$ becomes a
commutative ring with unit with respect to the \emph{twisted
product}
\begin{equation}\label{twprod}
  a\mathbin{*}b=a\cdot b+2[V^4]\cdot\partial a\cdot\partial b,
\end{equation}
where $\cdot$ denotes the product in $\OU$ and $V^4$ is a stably
complex manifold with $c_1^2[V^4]=-1$. One may take $V^4=\C
P^1\times\C P^1-\C P^2$ with the standard complex structure, or
$V^4=\overline{\C P\!}{\,}^2$ with the stably complex structure
defined by the isomorphism $\mathcal T\C
P^2\oplus\underline{\R}^2\cong\bar\eta\oplus\bar\eta\oplus\eta$.

We shall use the notation
\[
  m_i=\begin{cases}
    1&\text{if}\quad i+1\ne p^s\quad\text{for any prime }p,\\
    p&\text{if}\quad i+1=p^s\quad\text{for some prime $p$ and integer~$s>0$,}
  \end{cases}
\]
so that $[M^{2i}]\in\OU_{2i}$ represents a polynomial generator
whenever $s_i[M^{2i}]=\pm m_i$.

\begin{theorem}\label{Wring}\samepage
$\mathcal W$ is a polynomial ring on generators in every even
degree except~$4$:
\[
  \mathcal W\cong
  \Z[x_1,x_i\colon i>2],\quad x_1=[\C P^1],\quad\deg x_i=2i,
\]
with $s_i[x_i]=m_im_{i-1}$, and the boundary operator
$\partial\colon\mathcal W\to\mathcal W$, $\partial^2=0$, given by
\[
  \partial x_1=2, \quad \partial x_{2i}=x_{2i-1},
\]
satisfies the identity
\[
  \partial(a\mathbin{*} b)=a\mathbin{*}\partial b+
  \partial a\mathbin{*} b-
  x_1\mathbin{*}\partial a\mathbin{*}\partial b.
\]
\end{theorem}

The forgetful map $\alpha\colon\OSU\to\mathcal W$ is a ring
homomorphism; this follows from~\eqref{twprod} because
$\partial\alpha(x)=0$ for any $x\in\OSU$.

The fundamental result relating $\OSU$ and $\mathcal W$ is as
follows:

\begin{theorem}
There is an exact sequence of groups
\[
  0\longrightarrow\OSU_{2n-1}\stackrel\theta\longrightarrow
  \OSU_{2n}\stackrel\alpha\longrightarrow\mathcal W_{2n}
  \stackrel\beta\longrightarrow\OSU_{2n-2}
  \stackrel\theta\longrightarrow\OSU_{2n-1}\longrightarrow0,
\]
where $\theta$ is the multiplication by the generator
$\theta\in\OSU_1\cong\Z_2$, $\alpha$ is the forgetful
homomorphism, and $\alpha\beta=-\partial$.
\end{theorem}

Analysing the exact sequence above, one obtains the following
information about the free and torsion parts of~$\OSU$:

\begin{theorem}\label{freetors}\
\begin{itemize}
\item[(a)] $\mathop{\mathrm{Torsion}}(\OSU_n)=0$ unless $n=8k+1$
or $8k+2$, in which case $\mathop{\mathrm{Torsion}}(\OSU_n)$ is a
$\Z_2$-vector space of rank equal to the number of partitions
of~$k$.

\item[(b)] $\OSU_{2i}/\mathop{\mathrm{Torsion}}$ is isomorphic to
$\Ker(\partial\colon\mathcal W\to\mathcal W)$ if $2i\not\equiv
4\mod 8$ and is isomorphic to
$\mathop{\mathrm{Im}}(\partial\colon\mathcal W\to\mathcal W)$ if
$2i\equiv 4\mod 8$.

\item[(c)] There exist $SU$-bordism classes $w_{4k}\in\OSU_{8k}$, $k\ge1$, such that
$\mathop{\mathrm{Im}}\alpha/\mathop{\mathrm{Im}}\partial\cong\Z_2[w_{4k}]$.
Every torsion element of $\OSU$ is uniquely expressible in the
form $P\cdot\theta$ or $P\cdot\theta^2$ where $P$ is a polynomial
in $w_{4k}$ with coefficients $0$ or~$1$.
\end{itemize}
\end{theorem}

Note that we have
\begin{equation}\label{Wring2}
  \mathcal W\otimes\Z[{\textstyle\frac12}]\cong
  \Z[{\textstyle\frac12}][x_1,x_{2k-1},2x_{2k}-x_1x_{2k-1}\colon k>1],
\end{equation}
where $x_1^2=x_1\mathbin{*}x_1$ is a $\partial$-cycle, and each
$x_{2k-1}$, $2x_{2k}-x_1x_{2k-1}$ is a $\partial$-cycle.

\begin{theorem}\label{yidescr}
There exist elements $y_i\in\OSU_{2i}$, $i>1$, such that
$s_i(y_i)=m_im_{i-1}$ if $i$ is odd, $s_2(y_2)=-48$, and
$s_i(y_i)=2m_{i}m_{i-1}$ if $i$ is even and $i>2$. These elements
are mapped as follows under the forgetful homomorphism
$\alpha\colon\OSU\to\mathcal W$:
\[
  y_2\mapsto 2x_1^2,\quad y_{2k-1}\mapsto x_{2k-1},\quad
  y_{2k}\mapsto 2x_{2k}-x_1x_{2k-1},\quad k>1,
\]
where the $x_i$ are polynomial generators of $\mathcal W$. In
particular, $\OSU\otimes\Z[\frac12]$ embeds into~\eqref{Wring2} as
the polynomial subring generated by $x_1^2$, $x_{2k-1}$ and
$2x_{2k}-x_1x_{2k-1}$.
\end{theorem}

\subsection*{Quasitoric $SU$-manifolds} Omnioriented quasitoric manifolds whose
stably complex structures are $SU$ can be detected using the
following simple criterion:

\begin{proposition}[\cite{b-p-r10}]
An omnioriented quasitoric manifold $M$ has $c_1(M)=0$ if and only
if there exists a linear function $\varphi\colon\Z^n\to\Z$ such
that $\varphi(\lambda_i)=1$ for $i=1,\ldots,m$. Here the
$\lambda_i$ are the columns of matrix~\eqref{Lambdaqtoric}.

In particular, if some $n$ vectors of $\lambda_1,\ldots,\lambda_m$
form the standard basis $\mb e_1,\ldots,\mb e_n$, then $M$ is $SU$
if and only if the column sums of $\varLambda$ are all equal
to~$1$.
\end{proposition}
\begin{proof}
By Theorem~\ref{tangentqtoric}, $c_1(M)=v_1+\cdots+v_m$. By
Theorem~\ref{cohomqtoric}, $v_1+\cdots+v_m$ is zero in $H^2(M)$ if
and only if $v_1+\cdots+v_m=\sum_i\varphi(\lambda_i)v_i$ for some
linear function $\varphi\colon\Z^n\to\Z$, whence the result
follows.
\end{proof}

\begin{corollary}
A toric manifold $V$ cannot be~$SU$.
\end{corollary}
\begin{proof}
If $\varphi(\lambda_i)=1$ for all~$i$, then the vectors
$\lambda_i$ lie in the positive halfspace of~$\varphi$, so they
cannot span a complete fan.
\end{proof}

A more subtle result also rules out low-dimensional quasitoric
manifolds:

\begin{theorem}[{\cite[Theorem~6.13]{b-p-r10}}]\label{lowdimqt}
A quasitoric $SU$-manifold $M^{2n}$ represents $0$ in $\OU_{2n}$
whenever~$n<5$.
\end{theorem}

The reason for this is that the Krichever genus $\varphi_{\mathrm
K}\colon\OU\to R_{\mathrm K}$ vanishes on quasitoric
$SU$-manifolds, but $\varphi_{\mathrm K}$ is an isomorphism in
dimensions~$<10$.

Examples of quasitoric $SU$-manifolds representing nonzero bordism
classes in $\OU_{2n}$ for all $n\ge5$, except~$n=6$, were
constructed in~\cite{lu-wa}. We modify this construction to
present two particular families of quasitoric $SU$-manifolds
representing nonzero bordism classes in $\OU_{2n}$ for all
$n\ge5$, including~$n=6$.

\begin{construction}
Assume that $n_1=2k_1$ is positive even and $n_2=2k_2+1$ is
positive odd, and consider the manifold $L(n_1,n_2)$ from
Construction~\ref{Pn1n2}. We change the stably complex
structure~\eqref{TPn1n2} to the following:
\begin{multline*}
  \mathcal T
  L(n_1,n_2)\oplus\underline{\R}^4\\
  \cong
  \underbrace{p^*\bar\eta\oplus p^*\eta\oplus\cdots\oplus
  p^*\bar\eta\oplus p^*\eta}_{2k_1}\oplus p^*\bar\eta\oplus
  (\bar\gamma\otimes p^*\eta)\oplus
  \underbrace{\bar\gamma\oplus\gamma\oplus\cdots\oplus\bar\gamma\oplus\gamma}_{2k_2}
  \oplus\gamma
\end{multline*}
and denote the resulting stably complex manifold by $\widetilde
L(n_1,n_2)$. Its cohomology ring is given by the same
formula~\eqref{cohomproj2}, but
\begin{equation}\label{chernL}
  c\bigl(\widetilde
  L(n_1,n_2)\bigr)=(1-u^2)^{k_1}(1+u)(1+v-u)(1-v^2)^{k_2}(1-v),
\end{equation}
so $\widetilde L(n_1,n_2)$ is an $SU$-manifold of dimension
$2(n_1+n_2)=4(k_1+k_2)+2$.

\enlargethispage{2\baselineskip}
Viewing $L(n_1,n_2)$ as a quasitoric manifold with the
omniorientation coming from the complex structure, we see that
changing a line bundle $\rho_i$ in~\eqref{TMqtiso} to its
conjugate results in changing $\lambda_i$ to $-\lambda_i$
in~\eqref{Lambdaqtoric}. By applying this operation to the
corresponding columns of~\eqref{LambdaPn1n2} and then multiplying
from the left by an appropriate matrix from
$\mbox{\textit{GL}}\,(n,\Z)$, we obtain that $\widetilde
L(n_1,n_2)$ is the omnioriented quasitoric manifold over
$\varDelta^{n_1}\times\varDelta^{n_2}$ corresponding to the
matrix\\
\[
  \left(\begin{array}{ccccccccccc}
  \noalign{\vspace{-1\normalbaselineskip}}
  \multicolumn{5}{c}{\scriptstyle n_1=2k_1}\\[-5pt]
  \multicolumn{5}{c}{$\downbracefill$}\,\\
  1&0&0&\cdots&0&1                                           & & & & & \\
  0&1&0&\cdots&0&-1                                          & & & & & \\
  \vdots&\ddots&\ddots&\ddots&\vdots&\vdots    &&&\textrm{\huge 0} & & \\
  0&0&0&1&0&1                                                & & & & & \\
  0&0&0&0&1&-1                                               & & & & & \\
   & & & & &1       &                                    1&0&\cdots&0&1\\
   & & & & &        &                                    0&1&\cdots&0&-1\\
   & &\textrm{\huge 0} & & & &            \vdots&\ddots&\ddots&0&\vdots\\
   & & &                 & & &                                0&0&0&1&1\\[-5pt]
   & & & & &        & \multicolumn{4}{c}{$\upbracefill$}\\[-3pt]
   & & & & &        & \multicolumn{4}{c}{\scriptstyle n_2=2k_2+1}\\
  \noalign{\vspace{-1\normalbaselineskip}}
  \end{array}\right)
  \vspace{1\normalbaselineskip}
\]
The column sums of this matrix are~$1$ by inspection.
\end{construction}

\begin{construction}\label{constrN} The previous construction can be iterated by
considering projectivisations of sums of line bundles
over~$L(n_1,n_2)$. We shall need just one particular family of
this sort.

Given positive even $n_1=2k_1$ and odd $n_2=2k_2+1$, consider the
omnioriented quasitoric manifold $\widetilde N(n_1,n_2)$ over
$\varDelta^1\times\varDelta^{n_1}\times\varDelta^{n_2}$ with the
characteristic matrix\\
\[
  \left(\begin{array}{cccccccccccccc}
  1&1\\
   & & 1&0&0&\cdots&0&1                                           & & & & & &\\
   & & 0&1&0&\cdots&0&-1                                          & & & & & &\\
  \textrm{\huge 0}\!\!\!\! & & \vdots&\ddots&\ddots&\ddots&\vdots&\vdots    &&&&\textrm{\huge 0}  & &\\
   & & 0&0&0&1&0&1                                                & & & & & &\\
   & & 0&0&0&0&1&-1                                               & & & & & &\\[5pt]
   \noalign{\vspace{-1\normalbaselineskip}}
   & & \multicolumn{5}{c}{$\upbracefill$}\,\\[-3pt]
   &-1& \multicolumn{5}{c}{\scriptstyle n_1=2k_1} &0   &      1&0&0&\cdots&0&1\\
   &1& & & & & &0       &                                    0&1&0&\cdots&0&-1\\
   &0& & & & & &1       &                                    0&0&1&\cdots&0&1\\
   & & & &\textrm{\huge 0} & & & &  \vdots&\vdots&\vdots&\ddots&\vdots&\vdots\\
   & & & & &                 & & &                                0&0&0&0&1&1\\[-5pt]
   & & & & & & &        & \multicolumn{5}{c}{$\upbracefill$}\\[-3pt]
   & & & & & & &        & \multicolumn{5}{c}{\scriptstyle n_2=2k_2+1}\\
  \noalign{\vspace{-1\normalbaselineskip}}
  \end{array}\right)
  \vspace{1\normalbaselineskip}
\]
The column sums are $1$ by inspection, so $\widetilde N(n_1,n_2)$
is a quasitoric $SU$-manifold of dimension
$2(1+n_1+n_2)=4(k_1+k_2)+4$.

It can be seen that $\widetilde N(n_1,n_2)$ is a projectivisation
of a sum of $n_2+1$ line bundles over $\C P^1\times\C P^{n_1}$
with an amended stably complex structure.

The cohomology ring given by Theorem~\ref{cohomqtoric} is
\begin{equation}\label{cohomN}
  H^*(\widetilde N(n_1,n_2))\cong\Z[u,v,w]\big/
  \bigl(u^2,v^{n_1+1},(w-u)^2(v+w)w^{n_2-2}\bigr)
\end{equation}
with $uv^{n_1}w^{n_2}\langle\widetilde N(n_1,n_2)\rangle=1$. The
total Chern class is
\begin{equation}\label{chernN}
  c(\widetilde N(n_1,n_2))=(1-v^2)^{k_1}(1+v)
  (1-(w-u)^2)(1-v-w)(1-w^2)^{k_2-1}(1+w).
\end{equation}
\end{construction}

\subsection*{Quasitoric representatives for polynomial generators
of~{\rm$\OSU\otimes\Z[\frac12]$}} Our goal is to show that
elements $y_i\in\OSU_{2i}$ described in Theorem~\ref{yidescr} can
be represented by quasitoric $SU$-manifolds when $i\ge5$. This
will be done by calculating the characteristic numbers of
$\widetilde L(n_1,n_2)$ and $\widetilde N(n_1,n_2)$ and then
checking several divisibility conditions for binomial
coefficients. We shall need the following generalisation of Lucas'
Theorem:

\enlargethispage{2\baselineskip}
\begin{theorem}[{\cite[Theorem~1]{gran97}}]\label{granth}
Suppose that prime power $p^q$ and positive integers $m=n+r$ are
given. Write
\[
  n=n_0+n_1p+\cdots+n_{k-1}p^{k-1}+n_kp^k
\]
in base $p$, and let
\[
  N_j=n_j+n_{j+1}p+\cdots+n_{j+q-1}p^{q-1},\quad j\ge0.
\]
Also make the corresponding definitions for $m_j$, $M_j$, $r_j$,
$R_j$. Let $e_j$ be the number of indices $i\ge j$ for which
$n_i<m_i$. Then
\[
  \frac1{p^{e_0}}\binom nm\equiv(\pm1)^{e_{q-1}}
  \frac{N_0!_p}{M_0!_p\,R_0!_p}\cdot\frac{N_1!_p}{M_1!_p\,R_1!_p}
  \cdots\frac{N_k!_p}{M_k!_p\,R_k!_p}\mod p^q,
\]
where $\pm1$ is $-1$ except if $p=2$ and $q\ge3$, and $n!_p$
denotes the product of those positive integers $\le n$ that are not
divisible by~$p$.
\end{theorem}

\begin{lemma}\label{snL}
For $n_1=2k_1>0$ and $n_2=2k_2+1>0$, we have
\[
  s_{n_1+n_2}\bigl[\widetilde L(n_1,n_2)\bigr]=
  -\bin{n_1+n_2}1+\bin{n_1+n_2}2-\cdots-\bin{n_1+n_2}{n_1-1}+\bin{n_1+n_2}{n_1}.
\]
\end{lemma}
\begin{proof}
Using~\eqref{chernL} and~\eqref{cohomproj2} we calculate
\begin{align*}
  s_{n_1+n_2}\bigl(\widetilde L(n_1,n_2)\bigr)
  &=(v-u)^{n_1+n_2}+(k_2+1)(-1)^{n_1+n_2}v^{n_1+n_2}+k_2v^{n_1+n_2}\\
  &=(v-u)^{n_1+n_2}-v^{n_1+n_2}\\
  &=\Bigl(-\bin{n_1+n_2}1+\bin{n_1+n_2}2-\cdots
  -\bin{n_1+n_2}{n_1-1}+\bin{n_1+n_2}{n_1}\Bigr)u^{n_1}v^{n_2},
\end{align*}
and the result follows by evaluating at $\langle\widetilde
L(n_1,n_2)\rangle$.
\end{proof}

Note that $s_3(\widetilde L(2,1))=0$ in accordance with
Theorem~\ref{lowdimqt}. On the other hand, $s_{2+n_2}(\widetilde
L(2,n_2))\ne0$ for $n_2>1$, providing an example of a
non-bounding quasitoric $SU$-manifold in each dimension $4k+2$
with $k>1$.

\begin{lemma}\label{yL}
For $k>1$, there is a linear combination $y_{2k+1}$ of
$SU$-bordism classes $[\widetilde L(n_1,n_2)]$ with $n_1+n_2=2k+1$
such that $s_{2k+1}(y_{2k+1})=m_{2k+1}m_{2k}$.
\end{lemma}
\begin{proof}
By the previous lemma,
\[
  s_{n_1+n_2}\bigl[\widetilde L(n_1,n_2)-\widetilde L(n_1-2,n_2+2)\bigr]
  =\bin{n_1+n_2}{n_1}-\bin{n_1+n_2}{n_1-1}.
\]
The result follows from the next lemma.
\end{proof}

\begin{lemma}\label{gcddif}
For any integer $k>1$, we have
\[
  \gcd\bigl\{\bin{2k+1}{2i}-\bin{2k+1}{2i-1},\;0<i\le k\bigr\}
  =m_{2k+1}m_{2k}.
\]
\end{lemma}

\begin{remark}
This result has been obtained independently in a recent work of
Buchstaber and Ustinov on the coefficient rings of universal
formal group laws~\cite[\S9]{bu-us15}.
\end{remark}

\begin{proof}[Proof of Lemma~\ref{gcddif}]
We need to establish the following two facts:
\begin{itemize}
\item[(a)] The largest power of $2$ which divides each number
$\bin{2k+1}{2i}-\bin{2k+1}{2i-1}$ with $0<i\le k$ is $2$ if
$2k+2=2^s$ and is $1$ otherwise.
\item[(b)] The largest power of odd prime $p$ which divides each
number $\bin{2k+1}{2i}-\bin{2k+1}{2i-1}$ with $0<i\le k$ is $p$ if
$2k+1=p^s$ and is $1$ otherwise.
\end{itemize}
We prove (a) first.

\smallskip

\noindent\textbf{Case 1} ($2k+2=2^s$) Then $s>2$, as $k>1$. For
$0<i\le k$, we have
\[
  \bin{2k+1}{2i}-\bin{2k+1}{2i-1}\equiv\bin{2^s-1}{2i}+\bin{2^s-1}{2i-1}
  =\bin{2^s}{2i}\equiv0\mod2
\]
by Proposition~\ref{gcdbinom}. On the other hand,
\[
  \bin{2^s-1}{2}-\bin{2^s-1}{1}=(2^s-1)(2^{s-1}-1-1)
  =2(2^s-1)(2^{s-2}-1)\not\equiv 0\mod4.
\]

\smallskip

\noindent\textbf{Case 2} ($2k+2\ne2^s$) Write the base $2$
expansion
\[
  2k+2=n_12+\cdots+n_{l-1}2^{l-1}+2^l
\]
with $n_i=1$ or $0$. Set $2i=n_12+\cdots+n_{l-1}2^{l-1}$. We have
$2i\ne0$, as otherwise $2k+2=2^l$. Then
$\bin{2k+1}{2i}-\bin{2k+1}{2i-1}\equiv\bin{2k+2}{2i}\equiv1\mod2$
by Theorem~\ref{lucas}.

\smallskip

\noindent Now we prove~(b).

\smallskip

\noindent\textbf{Case 1} ($2k+1=p^s$) Then
$\bin{2k+1}{2i}-\bin{2k+1}{2i-1}\equiv 0\mod p$ for $0<i\le k$ by
Proposition~\ref{gcdbinom}. On the other hand, setting
$2i=p^{s-1}+1$, we get
\[
  \bin{2k+1}{2i}-\bin{2k+1}{2i-1}
  ={\textstyle\frac{2k-4i+2}{2i}}\bin{2k+1}{2i-1}
  ={\textstyle\frac{p^s-2p^{s-1}-1}{p^{s-1}+1}}\bin{p^s}{p^{s-1}}
  \not\equiv0\mod p^2.
\]
This follows from the fact that $p^s-2p^{s-1}-1>0$ as $k>1$, and
$\bin{p^s}{p^{s-1}}$ is not divisible by $p^2$ by Kummer's
theorem.

\smallskip

\noindent\textbf{Case 2} ($2k+1\ne p^s$) Write the base $p$
expansion
\[
  2k+1=n_0+n_1p+\cdots+n_{l-1}p^{l-1}+n_l p^l
\]
with $0\le n_i\le p-1$ (for $i=0,\ldots,l$) and $n_l>0$.

Assume that $n_0>1$. Then we set
\[
  2i=n_0+n_1p+\cdots+n_{l-1}p^{l-1}+(n_l-1)p^l.
\]
We have $\bin{2k+1}{2i}\equiv n_l\mod p$ by Theorem~\ref{lucas}.
Also,
\[
  2i-1=(n_0-1)+n_1p+\cdots+n_{l-1}p^{l-1}+(n_l-1)p^l>0
\]
and $\bin{2k+1}{2i-1}\equiv n_ln_0\mod p$. Therefore,
$\bin{2k+1}{2i}-\bin{2k+1}{2i-1}\equiv n_l(1-n_0)\not\equiv0\mod
p$.

Assume that $n_0=1$. Then we set $2i=2k$. We have
$\bin{2k+1}{2k}=2k+1\equiv1\mod p$ and
$\bin{2k+1}{2k-1}=k(2k+1)\equiv0\mod p$, so that
$\bin{2k+1}{2k}-\bin{2k+1}{2k-1}\not\equiv0\mod p$.

Finally, assume that $n_0=0$. Then we set
\[
  2i=n_0+n_1p+\cdots+n_{l-1}p^{l-1}+(n_l-1)p^l=
  n_qp^q+\cdots+n_{l-1}p^{l-1}+(n_l-1)p^l,
\]
where $q>0$ and $n_q>0$. We have $2i>0$, as otherwise $2k+1=p^l$.
Then $\bin{2k+1}{2i}\equiv n_l\mod p$. Also,
\[
  2i-1=(p-1)+(p-1)p+\cdots+(p-1)p^{q-1}+(n_q-1)p^q+\cdots+n_{l-1}p^{l-1}+(n_l-1)p^l,
\]
and $\bin{2k+1}{2i-1}\equiv0\mod p$ by Theorem~\ref{lucas}.
Therefore, $\bin{2k+1}{2i}-\bin{2k+1}{2i-1}\not\equiv0\mod p$.
\end{proof}

Now we turn our attention to the manifolds $\widetilde N(n_1,n_2)$
from Construction~\ref{constrN}.

\begin{lemma}\label{snN}
For $n_1=2k_1>0$ and $n_2=2k_2+1>0$, set $n=n_1+n_2+1$, so that
$\dim\widetilde N(n_1,n_2)=2n=4(k_1+k_2+1)$. Then
\[
  s_n\bigl[\widetilde N(n_1,n_2)\bigr]=2\bigl(-\bin{n}1+\bin{n}2-\cdots
  -\bin{n}{n_1-1}+\bin{n}{n_1}-n_1\bigr).
\]
\end{lemma}
\begin{proof}
Using~\eqref{chernN} and \eqref{cohomN} we calculate
\begin{multline}\label{snN1}
  s_{n}\bigl(\widetilde N(n_1,n_2)\bigr)
  =2(w-u)^n+(v+w)^n+(2k_2-1)w^n\\
  =2w^n-2nuw^{n-1}+w^n+\bin n1vw^{n-1}+\cdots+\bin n{2k_1}v^{2k_1}w^{2k_2+2}+(2k_2-1)w^n\\
  =-2nuw^{n-1}+(n-n_1)w^n+\bin n1vw^{n-1}+\cdots+\bin
  n{n_1}v^{n_1}w^{n-n_1}.
\end{multline}
Now we have to express each monomial above via $uv^{n_1}w^{n_2}$
using the identities in~\eqref{cohomN}, namely
\begin{equation}\label{3ident}
  u^2=0,\quad v^{n_1+1}=0,\quad
  w^{n_2+1}=2uw^{n_2}-vw^{n_2}+2uvw^{n_2-1}.
\end{equation}
We have
\begin{multline}\label{uwn-1}
  uw^{n-1}=uw^{n_1-1}w^{n_2+1}=uw^{n_1-1}(2uw^{n_2}-vw^{n_2}+2uvw^{n_2-1})
  \\=-uvw^{n-2}=\cdots=(-1)^juv^jw^{n-j-1}=\cdots=uv^{n_1}w^{n_2}.
\end{multline}
Also, we show that
\begin{equation}\label{vjwn-j}
  v^jw^{n-j}=(-1)^j2uv^{n_1}w^{n_2},\quad 0\le j\le n_1,
\end{equation}
by verifying the identity successively for $j=n_1,n_1-1,\ldots,0$.
Indeed, $v^{n_1}w^{n-n_1}=v^{n_1}w^{n_2+1}=2uv^{n_1}w^{n_2}$
by~\eqref{3ident}. Now, we have
\begin{multline*}
  v^{j-1}w^{n-j+1}=v^{j-1}w^{n_1+1-j}w^{n_2+1}
  =v^{j-1}w^{n_1+1-j}(2uw^{n_2}-vw^{n_2}+2uvw^{n_2-1})\\
  =2uv^{j-1}w^{n-j}-v^jw^{n-j}+2uv^jw^{n-1-j}=-v^jw^{n-j},
\end{multline*}
where the last identity holds because of~\eqref{uwn-1}. The
identity~\eqref{vjwn-j} is therefore verified completely.
Plugging~\eqref{uwn-1} and~\eqref{vjwn-j} into~\eqref{snN1} we
obtain
\[
  s_{n}\bigl(\widetilde N(n_1,n_2)\bigr)=\bigl(-2n+2(n-n_1)
  -2\bin{n}1+2\bin{n}2-\cdots
  -2\bin{n}{n_1-1}+2\bin{n}{n_1}\bigr)uv^{n_1}w^{n_2}.
\]
The result follows by evaluating at $\langle\widetilde
N(n_1,n_2)\rangle$.
\end{proof}

Note that $s_4(\widetilde N(2,1))=0$ in accordance with
Theorem~\ref{lowdimqt}. On the other hand, $s_n(\widetilde
N(2,n_2))=n^2-3n-4>0$ for $n>4$, providing an example of a
non-bounding quasitoric $SU$-manifold in each dimension $4k$
with $k>2$. This includes a 12-dimensional quasitoric
$SU$-manifold $\widetilde N(2,3)$, which was missing
in~\cite{lu-wa}.

\begin{lemma}\label{yN}
For $k>2$, there is a linear combination $y_{2k}$ of $SU$-bordism
classes $[\widetilde N(n_1,n_2)]$ with $n_1+n_2+1=2k$ such that
$s_{2k}(y_{2k})=2m_{2k}m_{2k-1}$.
\end{lemma}
\begin{proof}
The result follows from Lemma~\ref{snN} and
Lemmata~\ref{Nmod2},~\ref{Nmodp} below.
\end{proof}

\begin{lemma}\label{Nmod2}
For $k>2$, the largest power of $2$ which divides each number
\[
  a_i=-\bin{2k}1+\bin{2k}2-\cdots
  -\bin{2k}{2i-1}+\bin{2k}{2i}-2i,\quad 0<i< k,
\]
is $2$ if $2k=2^s$ and is $1$ otherwise.
\end{lemma}
\begin{proof}First assume that $2k=2^s$. Then $a_i\equiv0\mod
2$ by Proposition~\ref{gcdbinom}. On the other hand, we have
$a_1=-2^s+2^{s-1}(2^s-1)-2\not\equiv0\mod4$, because $s>2$.

Now assume that $2k\ne2^s$. We have
$a_i-a_{i-1}\equiv\bin{2k}{2i}\mod2$, so it is enough to find $i$
such that $\bin{2k}{2i}\ne0\mod2$. This was done in the proof of
Lemma~\ref{gcddif}.
\end{proof}

\begin{lemma}\label{Nmodp}
For $k>2$, the largest power of odd prime $p$ which divides each
\[
  a_i=-\bin{2k}1+\bin{2k}2+\cdots
  -\bin{2k}{2i-1}+\bin{2k}{2i}-2i,\quad 0<i< k,
\]
is $p$ if $2k+1=p^s$ and is $1$ otherwise.
\end{lemma}
\begin{proof}
Using the identity $2+\sum_{j=1}^{2k-1}(-1)^j\bin{2k}j=0$ we
obtain $a_{k-1}=0$ and
\begin{equation}\label{a+a}
  a_i+a_{k-i-1}=\bin{2k}{2i+1}-2k,\quad 0<i<k-1.
\end{equation}

\noindent\textbf{Case 1} ($2k+1=p^s$) We have
\[
  \bin{2k}{2i+1}=\bin{p^s-1}{2i+1}=\bin{p^s-1}{2i-1}
  \textstyle{\frac{(p^s-2i)(p^s-2i-1)}{2i(2i+1)}}\equiv-1\mod p
\]
by induction starting from $i=0$. Therefore,
\[
  a_i=a_{i-1}-\bin{2k}{2i-1}+\bin{2k}{2i}-2=a_{i-1}+
  {\textstyle\frac{p^s-4i}{2i}}\bin{p^s-1}{2i-1}-2\equiv0\mod p
\]
by induction starting from $a_0=0$. In view of~\eqref{a+a}, it
suffices to find $i$, where $0<i<k-1$, such that
$\bin{2k}{2i+1}-2k\not\equiv0\mod p^2$.

If $s=1$, then $p>5$ as $k>2$. We set $2i+1=3$, so that
\[
  \bin{2k}{2i+1}-2k=\bin{p-1}3-(p-1)={\textstyle\frac{p(p-1)(p-5)}6}
  \not\equiv0\mod p^2.
\]

Now assume that $s>1$. We set $2i+1=p^{s-1}$ and use
Theorem~\ref{granth} to calculate $\bin{2k}{2i+1}\mod p^2$. In the
notation of Theorem~\ref{granth}, we have $q=2$,
\begin{align*}
  &n=p^s-1=n_0+n_1p+\cdots+n_{s-2}p^{s-2}+n_{s-1}p^{s-1}\\
  &\qquad\qquad\ \;=(p-1)+(p-1)p+\cdots+(p-1)p^{s-2}+(p-1)p^{s-1},\\
  &\qquad N_0=\cdots=N_{s-2}=p^2-1,\quad N_{s-1}=p-1,\\[2pt]
  &m=p^{s-1}=m_0+m_1p+\cdots+m_{s-2}p^{s-2}+m_{s-1}p^{s-1}\\
  &\qquad M_0=\cdots=M_{s-3}=0,\quad M_{s-2}=p,\quad
    M_{s-1}=1,\\[2pt]
  &r=p^s-p^{s-1}-1=r_0+r_1p+\cdots+r_{s-2}p^{s-2}+r_{s-1}p^{s-1}\\
  &\qquad\qquad\qquad\quad\ \ =(p-1)+(p-1)p+\cdots+(p-1)p^{s-2}+(p-2)p^{s-1},\\
  &\qquad R_0=\cdots=R_{s-3}=p^2-1,\quad R_{s-2}=p^2-p-1,\quad
  R_{s-1}=p-2,
\end{align*}
and $e_0=e_1=0$. Therefore, Theorem~\ref{granth} gives
\[
  \bin{p^s-1}{p^{s-1}}\equiv{\textstyle\frac{(p^2-1)!_p}{p!_p\,(p^2-p-1)!_p}}
  \cdot{\textstyle\frac{(p-1)!_p}{1!_p\,(p-2)!_p}}=
  {\textstyle\frac{(p^2-1)\cdots(p^2-p+1)}{(p-1)!}}\cdot(p-1)\equiv p-1\mod
  p^2,
\]
and we obtain
\[
  \bin{2k}{2i+1}-2k=\bin{p^s-1}{p^{s-1}}-(p^s-1)\equiv p\mod p^2.
\]

\smallskip

\noindent\textbf{Case 2} ($2k+1\ne p^s$) In view of~\eqref{a+a},
it suffices to find $i$, where $0<i<k-1$, such that
$\bin{2k}{2i+1}-2k\not\equiv0\mod p$. Write the base $p$ expansion
\[
  2k=n_0+n_1p+\cdots+n_{l-1}p^{l-1}+n_l p^l
\]
with $0\le n_i\le p-1$ (for $i=0,\ldots,l$) and $n_l>0$. We have
$2k\equiv n_0\mod p$.

\smallskip

Assume that $n_0=0$. Then we set
\[
  2i+1=n_0+n_1p+\cdots+n_{l-1}p^{l-1}+(n_l-1) p^l.
\]
We have $\bin{2k}{2i+1}-2k\equiv n_l\not\equiv0\mod p$.

\smallskip

Assume that $0<n_0<p-2$. If $2k<p$, then $n_0=2k>5$. We set
$2i+1=3$, so that
\begin{equation}\label{3bino}
  \bin{2k}{2i+1}-2k\equiv\bin{n_0}3-n_0=\textstyle{\frac{n_0(n_0-4)(n_0+1)}6}
  \not\equiv0\mod p.
\end{equation}
If $2k>p$, then we set
\[
  2i+1=\begin{cases}n_0+1&\text{if $n_0$ is even},\\
                    n_0+2&\text{if $n_0$ is odd,}\end{cases}
\]
We have $2i+1<2k$ and $\bin{2k}{2i+1}-2k\equiv-n_0\not\equiv0\mod
p$.

\smallskip

Assume that $n_0=p-2$. If $p=3$, then $n_0=1$. We set $2i+1=5<2k$,
so that
$\bin{2k}{2i+1}-2k\equiv\bin{n_0}2\bin{n_1}1-1=-1\not\equiv0\mod
p$. If $p>3$, then we set $2i+1=3$, so that
$\bin{2k}{2i+1}-2k\not\equiv0\mod p$ by~\eqref{3bino}.


\smallskip

Assume that $n_0=p-1$ and $n_l<p-1$. Then we set
\[
  2i+1=n_0+n_1p+\cdots+n_{l-1}p^{l-1}+(n_l-1) p^l.
\]
We have $\bin{2k}{2i+1}-2k\equiv n_l-n_0\not\equiv0\mod p$.

\smallskip

Finally, assume that $n_0=p-1$ and $n_l=p-1$. As $2k\ne p^s-1$,
there exists $j$, where $0<j<l$, such that $n_j<p-1$. Then we set
\[
  2i+1=n_0+n_1p+\cdots+n_{j-1}p^{j-1}+(n_j+1)p^j.
\]
We have $2i+1<2k$ and $\bin{2k}{2i+1}-2k\equiv-n_0\not\equiv0\mod
p$.
\end{proof}

We now can prove our main result:

\begin{theorem}\label{mainth}
There exist quasitoric $SU$-manifolds $M^{2i}$, $i\ge5$, with
$s_i(M^{2i})=m_im_{i-1}$ if $i$ is odd and
$s_i(M^{2i})=2m_{i}m_{i-1}$ if $i$ is even. These quasitoric
manifolds represent polynomial generators
of~$\OSU\otimes\Z[\frac12]$.
\end{theorem}
\begin{proof}
It follows from Lemmata~\ref{yL} and~\ref{yN} that there exist
linear combinations of $SU$-bordism classes represented by
quasitoric $SU$-manifolds with the required properties. We observe
that application of Construction~\ref{diamondsum} to two
quasitoric $SU$-manifolds $M$ and $M'$ produces a quasitoric
$SU$-manifold representing their bordism sum. Also, the
$SU$-bordism class $-[M]$ can be represented by the omnioriented
qua\-si\-toric $SU$-manifold obtained by reversing the global
orientation of~$M$. Therefore, we can replace the linear
combinations obtained using Lemmata~\ref{yL} and~\ref{yN} by
appropriate connected sums, which are quasitoric $SU$-manifolds.
\end{proof}

\subsection*{Concluding remarks}
By analogy with Theorem~\ref{6.11}, we may ask the following:

\begin{question}\label{question}
Which $SU$-bordism classes of dimension $>8$ can be represented by
quasitoric $SU$-manifolds?
\end{question}

Theorem~\ref{mainth} provides quasitoric representatives for the
elements $y_i\in\OSU_{2i}$ described in Theorem~\ref{yidescr} for
$i\ge5$. The elements $y_2$, $y_3$, $y_4$ cannot be represented by
quasitoric manifolds because of Theorem~\ref{lowdimqt}. Any
polynomial in these elements cannot be represented by a quasitoric
manifold for the same reason: the Krichever genus $\varphi_{\mathrm
K}\colon\OU\to R_{\mathrm K}$ vanishes on quasitoric
$SU$-manifolds, but $\varphi_{\mathrm K}$ is nonzero on any
polynomial in $y_2$, $y_3$, $y_4$. We thank Michael Wiemeler for
this observation.


The element $x_1^2\in\mathcal W_4$ (see Theorem~\ref{Wring}) is
represented by $9\C P^1\times\C P^1-8\C P^2$, which is also the
bordism class of a toric manifold over a 12-gon, with
characteristic numbers $c_1^2=0$ and $c_2=12$ (so $s_2=-24$). The
element $y_2=2x_1^2\in\OSU_4$ is represented by a \emph{K3
surface}, but not by a toric manifold.

The 6-sphere $S^6$ has a $T^2$-invariant almost complex structure
as the homogeneous space $G_2/SU(3)$ of the exceptional Lie group
$G_2$ (see~\cite{bo-hi58}), and therefore represents an
$SU$-bordism class in~$\OSU_6$. Its characteristic numbers are
$c_1^3=c_1c_2=0$ and $c_3=2$. Therefore, $s_3[S^6]=6=m_3m_2$, so
$S^6$ represents $y_3\in\OSU_6$.

It would be interesting to find good geometric representatives for
$y_4\in\OSU_8$, and also for the elements $w_{4k}\in\OSU_{8k}$
that control the 2-torsion in $\OSU$ (see Theorem~\ref{freetors}~(c)).
The image of $w_{4k}$ under the forgetful homomorphism
$\alpha\colon \OSU_{8k}\to \mathcal W_{8k}$ is $x_1^{4k}$, so it
is decomposable in~$\OU$ and has $s_{4k}[w_{4k}]=0$. The
conditions on the characteristic numbers specifying $w_{4k}$ are
given in~\cite[(19.3)]{co-fl66m}.


\end{document}